\documentclass[12pt,english]{amsart}
\usepackage[T1]{fontenc}
\usepackage[latin9]{inputenc}
\usepackage{geometry}
\usepackage{units}
\usepackage{amsthm}
\usepackage{amstext}
\usepackage{amssymb}
\PassOptionsToPackage{normalem}{ulem}
\usepackage{ulem}

\makeatletter
\numberwithin{equation}{section}
\numberwithin{figure}{section}
\theoremstyle{plain}
\newtheorem{thm}{\protect\theoremname}
  \theoremstyle{plain}
  \newtheorem{lem}[thm]{\protect\lemmaname}
 \theoremstyle{definition}
 \newtheorem*{defn*}{\protect\definitionname}
  \theoremstyle{plain}
  \newtheorem{prop}[thm]{\protect\propositionname}

\makeatother

\usepackage{babel}
  \providecommand{\definitionname}{Definition}
  \providecommand{\lemmaname}{Lemma}
  \providecommand{\propositionname}{Proposition}
\providecommand{\theoremname}{Theorem}

\begin{document}

\title{THE AUTOMORPHISM GROUP FOR $p$-CENTRAL $p$-GROUPS}

\author{Anitha Thillaisundaram}

\address{Magdalene College, Cambridge, CB3 0AG, United Kingdom.}

\email{anitha.t@cantab.net}

\date{24th September 2011}
\begin{abstract}
A $p$-group $G$ is $p$-central if $G^{p}\le Z(G)$, and $G$ is
$p^{2}$-abelian if $(xy)^{p^{2}}=x^{p^{2}}y^{p^{2}}$ for all $x,y\in G$.
We prove that for $G$ a finite $p^{2}$-abelian $p$-central $p$-group,
excluding certain cases, the order of $G$ divides the order of $\text{Aut}(G)$.
\end{abstract}
\maketitle

\section*{Introduction}

$\ $

The following conjecture has been the subject of much debate over
the past forty years or so.

For $G$ a group, we denote the group of\emph{ }automorphisms of $G$
by $\text{Aut}(G)$. Here $|G|$ denotes the order of the group $G$.

$\ $

\textbf{Conjecture A. }\cite{key-4} \emph{For $G$ a non-cyclic $p$-group
of order $p^{n}$ with $n\ge3$, $|G|$ divides $|\text{Aut}(G)|$. }

$\ $

Results in favour of Conjecture A have been made by Buckley \cite{key-Buckley};
Davitt \cite{key-28,key-41,key-29,key-24}; Exarchakos \cite{key-39};
Faudree \cite{key-39}; Fouladi, Jamali \& Orfi \cite{key-25}; Gasch\"{u}tz
\cite{key-40}; Gavioli \cite{key-43}; Hummel \cite{key-27}; Otto
\cite{key-29,key-24,key-30}; and Yadav \cite{key-4}. See Result
B below for further details. I apologise if I have unknowingly omitted
other references.

$\ $

Notice that each non-central element $g$ of $G$ induces a non-trivial
automorphism of $G$ via conjugation. This defines an \emph{inner
automorphism} of $G$. $\text{Inn}(G)$ denotes the subgroup of inner
automorphisms of $G$, which is normal in $\text{Aut}(G)$. The non-inner
automorphisms are called \emph{outer automorphisms. }They are elements
in $\nicefrac{\text{Aut}(G)}{\text{Inn}(G)}$ and so are defined modulo
$\text{Inn}(G)$. We denote the group of outer automorphisms of $G$
by $\text{Out}(G)$.

Certainly as $\text{Inn}(G)\cong\nicefrac{G}{Z(G)}$, we can rephrase
the question to whether or not $|Z(G)|$ divides $|\text{Out}(G)|$.

Conjecture A has been established to be true for several classes of
$p$-groups, as listed below in Result B. We note that a group $G$
is the \emph{central product} of two subgroups $H$ and $K$ if (a)
$[H,K]=1$, (b) $G=HK$, and (c) $H\cap K=Z(G)$; and a group is \emph{modular
}if each subgroup commutes with every other subgroup, i.e. for $H,K\le G$,
we have $HK=KH$.

$\ $

\textbf{Result B. }Conjecture A holds for the following finite $p$-groups:
\begin{itemize}
\item $p$-abelian $p$-groups \cite{key-28};
\item $p$-groups of class 2 \cite{key-39};
\item $p$-groups of maximal class (or coclass 1) \cite{key-30}; 
\item $p$-groups of coclass 2 \cite{key-25};
\item $p$-groups with centre of order $p$ \cite{key-40};
\item $p$-groups of order at most $p^{7}$ \cite{key-41,key-42,key-43}; 
\item modular $p$-groups \cite{key-29};
\item $p$-groups with $\nicefrac{G}{Z(G)}$ metacyclic \cite{key-24};
\item $p$-groups with $|\nicefrac{G}{Z(G)}|\le p^{4}$ \cite{key-24};
\item $G=A\times B$ where $A$ is abelian and $|B|$ divides $|\text{Aut}(B)|$
\cite{key-30}; 
\item $G$ a central product of non-trivial subgroups $H$ and $A$, where
$A$ is abelian and $|H|$ divides $|\text{Aut}(H)|$ \cite{key-27};
\item $G$ with a non-trivial normal subgroup $N$ such that $N\cap[G,G]=1$,
and $|\nicefrac{G}{N}|$ divides $|\text{Aut}(\nicefrac{G}{N})|$
\cite{key-Buckley};
\item $G$ such that $xZ(G)\subseteq x^{G}$ for all $x\in G\backslash Z(G)$,
where $x^{G}$ denotes the conjugacy class of $x$ in $G$ \cite{key-4}.
\end{itemize}
$\ $

For $n\in\mathbb{N}$, we have $G^{\left\{ n\right\} }=\left\{ x^{n}|x\in G\right\} $
and $G^{n}=\langle G^{\left\{ n\right\} }\rangle$. 

We define the \emph{centre} of $G$ as $Z(G)=\left\{ x\in G|x^{-1}gx=g\ \text{for all }g\in G\right\} $. 

We say that a $p$-group $G$ is $p$-central if $G^{p}\le Z(G)$.
We define $G$ to be $p^{2}$-abelian if for all $x,y\in G$, we have
$(xy)^{p^{2}}=x^{p^{2}}y^{p^{2}}$.

In this paper, we prove the conjecture for $p^{2}$-abelian $p$-central
$p$-groups.
\begin{thm}
\label{thm:Conj} For $p$ an odd prime, let $G$ be a non-abelian
$p^{2}$-abelian $p$-central $p$-group, with $|G|\ge p^{3}$. Suppose
that the centre $Z(G)$ of $G$ is of the form
\[
Z(G)\cong\nicefrac{\mathbb{Z}}{p^{e_{1}}\mathbb{Z}}\times\ldots\times\nicefrac{\mathbb{Z}}{p^{e_{n}}\mathbb{Z}}
\]
where $3\le e_{1}\le\ldots\le e_{n}$ and $n\ge3$. Then $|G|$ divides
$|\text{Aut}(G)|$.
\end{thm}
As $p$-abelian groups are $p^{2}$-abelian, Theorem \ref{thm:Conj}
partially generalizes the fact that $p$-abelian $p$-groups satisfy
Conjecture A.

In the following, we prove Theorem \ref{thm:Conj}, using results
on extending automorphisms of subgroups (by Passi, Singh \& Yadav
\cite{key-38}) and on counting automorphisms of abelian $p$-groups
(by Hillar \& Rhea \cite{key-36}).

$\ $

This paper is an extract from my PhD thesis under the supervision
of Rachel Camina. 

\textbf{Acknowledgements.} I am very grateful to Rachel Camina for
her time and helpful comments. Also thank you to Chris Brookes and
to Gavin Armstrong for a thorough reading of this work. 

This research has been made possible through the generous support
from the Cambridge Commonwealth Trust, the Cambridge Overseas Research
Scholarship, and the Leslie Wilson Scholarship (from Magdalene College,
Cambridge).

$\ $

\section*{Proof of Theorem \ref{thm:Conj}}

$\ $

Let $G$ be a $p^{2}$-abelian $p$-central $p$-group and denote
$Z(G)$ by simply $Z$. Let $|\text{Out}(G)|_{p}$ denote the largest
power of $p$ that divides $|\text{Out}(G)|$. This corresponds to
the order of a Sylow $p$-subgroup of $\text{Out}(G)$. 

To prove the theorem, as $|\text{Inn}(G)|=|\nicefrac{G}{Z}|$, it
suffices to show that
\[
|\text{Out}(G)|_{p}\ge|Z|.
\]

Let
\[
E:1\rightarrow N\begin{array}{c}
\rightarrow\end{array}G\begin{array}{c}
\rightarrow\end{array}Q\rightarrow1
\]
be an extension of the group $N$ by the group $Q$ .

Note: if $N\le Z$, then $E$ is termed a \emph{central} extension.

$\ $

Here $\text{Aut}^{\nicefrac{G}{Z}}(G)$ is the subgroup of automorphisms
of $G$ that induce the identity on $\nicefrac{G}{Z}$. For $N$ normal
in $G$, we define $\text{Aut}_{N}(G)$ to be the subgroup of automorphisms
of $G$ that normalize $N$.

Our plan is to compute a lower bound for the size of a Sylow $p$-subgroup
of $\text{Out}(G)$. To do this, we choose to consider elements of
$\text{Aut}(Z)$ that extend to elements of $\text{Aut}^{\nicefrac{G}{Z}}(G)$.
Naturally any such extension of a non-identity automorphism of $Z$
is non-inner.

Such extendable elements of $\text{Aut}(Z)$ are determined by the
following result. In the following, $t:Q\rightarrow G$ is a left
transversal, and $\mu:Q\times Q\rightarrow N$ is defined by
\[
t(xy)\mu(x,y)=t(x)t(y).
\]
Also, $N^{Q}$ denotes the group of all maps $\psi$ from $Q$ to
$N$ such that $\psi(1)=1$.
\begin{lem}
\label{Lemma 2} \cite{key-38} Let $1\rightarrow N\rightarrow G\rightarrow Q\rightarrow1$
be a central extension. Using the notation above, if $\gamma\in\text{Aut}_{N}(G)$,
then there exists a triplet $(\theta,\phi,\chi)\in\text{Aut}(N)\times\text{Aut}(Q)\times N^{Q}$
such that for all $x,y\in Q$ and $n\in N$ the following conditions
are satisfied:

(1) $\gamma(t(x)n)=t(\phi(x))\chi(x)\theta(n)$,

(2) $\mu(\phi(x),\phi(y))\theta(\mu(x,y)^{-1})=\chi(x)^{-1}\chi(y)^{-1}\chi(xy)$.

Conversely, if $(\theta,\phi,\chi)\in\text{Aut}(N)\times\text{Aut}(Q)\times N^{Q}$
is a triplet satisfying equation (2), then $\gamma$ defined by (1)
is an automorphism of $G$ normalizing $N$. 
\end{lem}
$\ $ 

We take $N=Z$, $Q=\nicefrac{G}{Z}$ in the lemma above. It is clear
from equation (1) of Lemma \ref{Lemma 2} that when $\phi=1$, the
automorphism $\gamma$ induces the identity on $\nicefrac{G}{Z}$.
So we set $\phi=1$. Given a suitable $\theta\in\text{Aut}(Z)$, we
construct the required $\chi$ to satisfy: 
\begin{equation}
\mu(x,y)\theta(\mu(x,y)^{-1})=\chi(x)^{-1}\chi(y)^{-1}\chi(xy).\label{eq:Star}
\end{equation}
$\ $

In \cite{key-36}, Hillar and Rhea give a useful description of the
automorphism group of an arbitrary abelian $p$-group, and they compute
the size of this automorphism group. We sketch their results here.
The first complete characterization of the automorphism group of an
abelian group was, however, given by Ranum \cite{key-Ranum}.

We will use Hillar and Rhea's account to characterize $\text{Aut}(Z)$.
First, we set up the relevant notation and results leading to our
desired description.

We begin with an arbitrary abelian $p$-group $H_{p}$, where

\[
H_{p}\cong\nicefrac{\mathbb{Z}}{p^{e_{1}}\mathbb{Z}}\times\ldots\times\nicefrac{\mathbb{Z}}{p^{e_{n}}\mathbb{Z}}
\]
and $1\le e_{1}\le\ldots\le e_{n}$ are positive integers.

Hillar and Rhea first describe $\text{End}(H_{p})$, the endomorphism
ring of $H_{p}$, as a quotient of a matrix subring of $\mathbb{Z}^{n\times n}$.
Then, as we will see below, the units $\text{Aut}(H_{p})\subseteq\text{End}(H_{p})$
are characterized from this description. 

An element of $H_{p}$ is represented by a column vector $(\alpha_{1},\ldots,\alpha_{n})^{T}$
where $\alpha_{i}\in\nicefrac{\mathbb{Z}}{p^{e_{i}}\mathbb{Z}}$.
\begin{defn*}
(\cite{key-36}, Definition 3.1) 
\[
R_{p}=\{(a_{ij})\in\mathbb{Z}^{n\times n}:p^{e_{i}-e_{j}}|a_{ij}\text{ for all }i\text{ and }j\text{ satisfying }1\le j\le i\le n\}.
\]

\end{defn*}
From \cite{key-36}, we have that $R_{p}$ forms a ring. 

Let $\pi_{i}:\mathbb{Z}\longrightarrow\nicefrac{\mathbb{Z}}{p^{e_{i}}\mathbb{Z}}$
be defined by $x\mapsto x$ mod $p^{e_{i}}$. Let $\pi:\mathbb{Z}^{n}\longrightarrow H_{p}$
be the homomorphism given by 
\[
\pi(x_{1},\ldots,x_{n})^{T}=(\pi_{1}(x_{1}),\ldots,\pi_{n}(x_{n}))^{T}.
\]
Here is the description of $\text{End}(H_{p})$ as a quotient of the
matrix ring $R_{p}$.
\begin{thm}
\label{thm:3.3}\cite{key-36} The map $\psi:R_{p}\longrightarrow\text{End}(H_{p})$
given by
\[
\psi(A)(\alpha_{1},\ldots,\alpha_{n})^{T}=\pi(A(\alpha_{1},\ldots,\alpha_{n})^{T})
\]
is a surjective ring homomorphism.
\end{thm}
Let $K$ be the set of matrices $A=(a_{ij})\in R_{p}$ such that $p^{e_{i}}|a_{ij}$
for all $i,j$. This forms an ideal.
\begin{lem}
\label{lem:3.4}\cite{key-36} The ideal $K$, as defined above, is
the kernel of $\psi$.
\end{lem}
Theorem \ref{thm:3.3} and Lemma \ref{lem:3.4} give that $\text{End}(H_{p})$
is isomorphic to $\nicefrac{R_{p}}{K}$. For more details, the reader
is referred to \cite{key-36}.

The following is a complete description of $\text{Aut}(H_{p})$.
\begin{thm}
(\cite{key-36}, Theorem 3.6) An endomorphism $M=\psi(A)$ is an automorphism
if and only if $A(\text{mod }p)\in GL_{n}(\mathbb{F}_{p})$.
\end{thm}
Hillar and Rhea illustrate how to calculate $|\text{Aut}(H_{p})|$,
which is presented in the theorem below. First, the following numbers
are defined:

\[
d_{k}=\max\{m:e_{m}=e_{k}\},\ c_{k}=\min\{m:e_{m}=e_{k}\}.
\]
Since $e_{m}=e_{k}$ for $m=k$, we have the two inequalities $d_{k}\ge k$
and $c_{k}\le k$. 

Note that 
\[
c_{1}=c_{2}=\ldots=c_{d_{1}},
\]
and 
\[
c_{d_{1}+1}=\ldots=c_{d_{d_{1}+1}},
\]
etc. So we have
\[
c_{1}=\ldots=c_{d_{1}}<c_{d_{1}+1}=\ldots=c_{d_{d_{1}+1}}<c_{d_{d_{1}+1}+1}=\ldots.
\]
We introduce the numbers $e_{i}',C_{i},D_{i}$ as follows. Define
the set of distinct numbers $\{e_{i}'\}$ such that 
\[
\{e_{i}'\}=\{e_{j}\}\text{ and }e_{1}'<e_{2}'<\ldots.
\]
Let $l\in\mathbb{N}$ be the size of $\{e_{i}'\}$. So $e_{1}'=e_{1}$,
$e_{2}'=e_{d_{1}+1}$, $\ldots$ , $e_{l}'=e_{n}$. 

Now define

\[
D_{i}=\max\{m:e_{m}=e_{i}'\}\text{ for }1\le i\le l
\]
and 
\[
C_{i}=\min\{m:e_{m}=e_{i}'\}\text{ for }1\le i\le l.
\]
Note that $C_{1}=1$ and $D_{l}=n$. For convenience, we also define
$C_{l+1}=n+1$.
\begin{thm}
\label{thm:abelian formula}(\cite{key-36}, Theorem 4.1) The abelian
group $H_{p}=\mathbb{Z}/p^{e_{1}}\mathbb{Z}\times\ldots\times\mathbb{Z}/p^{e_{n}}\mathbb{Z}$
has
\[
|\text{Aut}(H_{p})|=\prod_{k=1}^{n}(p^{d_{k}}-p^{k-1})\prod_{j=1}^{n}(p^{e_{j}})^{n-d_{j}}\prod_{i=1}^{n}(p^{e_{i}-1})^{n-c_{i}+1}.
\]
\end{thm}
\begin{proof}
Their calculation involves finding all elements of $R_{p}$ that are
invertible modulo $p$, and computing the distinct ways of extending
such elements to automorphisms of the group. 

So, we need to count all matrices $M\in R_{p}$ that are invertible
modulo $p$. These $M$ are {}``upper block triangular'' matrices
which may be expressed in the following three forms.
\[
M=\left[\begin{array}{cccccccccc}
m_{11} &  &  &  &  &  &  &  &  & *\\
\vdots\\
m_{D_{1}1} & \cdots & m_{D_{1}D_{1}}\\
 &  &  & m_{C_{2}C_{2}}\\
 &  &  & \vdots\\
 &  &  & m_{D_{2}C_{2}} & \cdots & m_{D_{2}D_{2}}\\
 &  &  &  &  &  & \ddots\\
 &  &  &  &  &  &  & m_{C_{l}C_{l}}\\
 &  &  &  &  &  &  & \vdots\\
0 &  &  &  &  &  &  & m_{D_{l}C_{l}} & \cdots & m_{D_{l}D_{l}}
\end{array}\right]
\]
or
\[
M=\left[\begin{array}{cccc}
m_{11} & m_{12} & \cdots & m_{1n}\\
\vdots\\
m_{d_{1}1}\\
 & m_{d_{2}2}\\
 &  & \ddots\\
0 &  &  & m_{d_{n}n}
\end{array}\right]=\left[\begin{array}{cccccc}
m_{1c_{1}} &  &  &  &  & *\\
 & m_{2c_{2}}\\
 &  & \ddots\\
0 &  &  & m_{nc_{n}} & \cdots & m_{nn}
\end{array}\right].
\]
The number of such $M$ is
\[
\prod_{k=1}^{n}(p^{d_{k}}-p^{k-1}),
\]
since we require linearly independent columns. 

So the first step to calculating $|\text{Aut}(H_{p})|$ is done. The
second half of the computation is to count the number of extensions
of $M$ to $\text{Aut}(H_{p})$. To extend each entry $m_{ij}$ from
$m_{ij}\in\nicefrac{\mathbb{Z}}{p\mathbb{Z}}$ to $a_{ij}\in\nicefrac{p^{e_{i}-e_{j}}\mathbb{Z}}{p^{e_{i}}\mathbb{Z}}$
(if $e_{i}>e_{j}$), or $a_{ij}\in\nicefrac{\mathbb{Z}}{p^{e_{i}}\mathbb{Z}}$
(if $e_{i}\le e_{j}$), such that
\[
a_{ij}\equiv m_{ij}\ (\text{mod }p),
\]
we have $p^{e_{j}}$ ways to do so for the necessary zeros (that is,
when $e_{i}>e_{j}$), as any element of $\nicefrac{p^{e_{i}-e_{j}}\mathbb{Z}}{p^{e_{i}}\mathbb{Z}}$
works. Similarly, there are $p^{e_{i}-1}$ ways for the not necessarily
zero entries (that is, when $e_{i}\le e_{j})$, as any element of
$\nicefrac{p\mathbb{Z}}{p^{e_{i}}\mathbb{Z}}$ will do.
\end{proof}
$\ $

\emph{We apply Hillar and Rhea's method to $M=I_{n\times n}$}. We
consider all extensions of $I_{n\times n}$ to $\text{Aut}(Z)$. Using
Lemma \ref{Lemma 2}, we identify which of these elements of $\text{Aut}(Z)$
can be extended to $\text{Aut}^{\nicefrac{G}{Z}}(G)$.

To this end, we prove the following.
\begin{prop}
\label{pro:(a),(b),(c)}Let $G$ be a finite non-abelian $p^{2}$-abelian
$p$-central $p$-group. Suppose $Z=Z(G)\cong\nicefrac{\mathbb{Z}}{p^{e_{1}}\mathbb{Z}}\times\nicefrac{\mathbb{Z}}{p^{e_{2}}\mathbb{Z}}\times\ldots\times\nicefrac{\mathbb{Z}}{p^{e_{n}}\mathbb{Z}}$
where $2\le e_{1}\le e_{2}\le\ldots\le e_{n}$ and $n\in\mathbb{N}$.
Let $\theta\in\text{Aut}(Z)$ be such that:

(a) $\theta$ is represented as a matrix $A\in R_{p}$;

(b) $A(\text{mod }p)\equiv I_{n\times n}$;

(c) $a_{ij}\equiv0$ mod $p^{e_{i}-e_{j}+2}$ for $i\ne j$ with $e_{i}\ge e_{j}$,
and $a_{ii}\equiv1$ mod $p^{2}$. 

Then $\theta$ can be extended to $\widetilde{\theta}\in\text{Aut}^{\nicefrac{G}{Z}}(G)$.\end{prop}
\begin{proof}
By Lemma \ref{Lemma 2}, we know that $\theta\in\text{Aut}(Z)$ can
be extended to $\text{Aut}(G)$ if there exist $\phi$ and $\chi$
such that condition $(2)$ of the lemma holds. Our strategy is to
take $\phi=1$ and to construct a suitable $\chi$. 

Recall equation (\ref{eq:Star}):

\[
\mu(x,y)\theta(\mu(x,y)^{-1})=\chi(x)^{-1}\chi(y)^{-1}\chi(xy).
\]
We aim to construct $\chi$ such that equation (\ref{eq:Star}) is
satisfied. 

We express $Z$ as 
\[
\langle z_{1}\rangle\times\langle z_{2}\rangle\times\ldots\times\langle z_{n}\rangle\cong\nicefrac{\mathbb{Z}}{p^{e_{1}}\mathbb{Z}}\times\nicefrac{\mathbb{Z}}{p^{e_{2}}\mathbb{Z}}\times\ldots\times\nicefrac{\mathbb{Z}}{p^{e_{n}}\mathbb{Z}}
\]
where $\{z_{1},\ldots,z_{n}\}$ generates $Z$.

Before we prove that a general $\theta\in\text{Aut}(Z)$ which satisfies
(a) to (c) can be extended to $\widetilde{\theta}$ in $\text{Aut}(G)$,
we illustrate our method by considering the following automorphism
$\theta_{0}$ (in its matrix representation): 

\[
\varphi(\theta_{0})=A_{0}=\left(\begin{array}{cccc}
1+p^{2} &  &  & 0\\
 & 1+p^{2}\\
 &  & \ddots\\
0 &  &  & 1+p^{2}
\end{array}\right).
\]
The automorphism $\theta_{0}$ clearly satisfies our conditions (a)
to (c).

Writing $\mu(x,y)\in Z$ as $z_{1}^{\alpha_{1}}z_{2}^{\alpha_{2}}\ldots z_{n}^{\alpha_{n}}$
for $\alpha_{i}\in\nicefrac{\mathbb{Z}}{p^{e_{i}}\mathbb{Z}}$, we
have that $\theta_{0}(\mu(x,y))$ is given by
\[
A_{0}\cdot\left(\begin{array}{c}
\alpha_{1}\\
\alpha_{2}\\
\vdots\\
\alpha_{n}
\end{array}\right)=\left(\begin{array}{cccc}
1+p^{2} &  &  & 0\\
 & 1+p^{2}\\
 &  & \ddots\\
0 &  &  & 1+p^{2}
\end{array}\right)\cdot\left(\begin{array}{c}
\alpha_{1}\\
\alpha_{2}\\
\vdots\\
\alpha_{n}
\end{array}\right)
\]
\[
=\left(\begin{array}{c}
(1+p^{2})\alpha_{1}\\
(1+p^{2})\alpha_{2}\\
\vdots\\
(1+p^{2})\alpha_{n}
\end{array}\right),
\]
which translates to 
\[
z_{1}^{(1+p^{2})\alpha_{1}}\ldots z_{n}^{(1+p^{2})\alpha_{n}}.
\]
The left-hand side of (\ref{eq:Star}) is then
\[
(z_{1}^{-\alpha_{1}}z_{2}^{-\alpha_{2}}\ldots z_{n}^{-\alpha_{n}})^{p^{2}}=(\mu(x,y)^{-1})^{p^{2}}.
\]

As $\mu(x,y)=t(xy)^{-1}t(x)t(y)$, we see that by the $p^{2}$-abelian
and the $p$-central properties, 
\begin{equation}
\mu(x,y)^{p^{2}}=t(x)^{p^{2}}t(y)^{p^{2}}t(xy)^{-p^{2}}.\label{eq:dagger}
\end{equation}
So setting $\chi_{0}(x)=t(x)^{p^{2}}$ (and $\theta_{0}(z)=z^{1+p^{2}}$)
works as (\ref{eq:Star}) is fulfilled. Note that $\chi_{0}$ is defined
on $Q$ as required. Thus $\theta_{0}$ can be extended to $\widetilde{\theta_{0}}\in\text{Aut}^{\nicefrac{G}{Z}}(G)$. 

We now consider a general element $\theta$ satisfying conditions
(a) to (c). We may express $\theta$ as the matrix $A$ below:

\[
A=\left(\begin{array}{cccc}
1+s_{1}p^{r_{1}} & a_{12} & \ldots & a_{1n}\\
a_{21} & 1+s_{2}p^{r_{2}} &  & \vdots\\
\vdots &  & \ddots & a_{n-1,n}\\
a_{n1} & \ldots & a_{n,n-1} & 1+s_{n}p^{r_{n}}
\end{array}\right)
\]
where 
\[
r_{i}\ge2\text{ and }s_{i}\in\nicefrac{\mathbb{Z}}{p^{e_{i}-r_{i}}\mathbb{Z}},
\]
and for $i\ne j$, 
\[
a_{ij}\equiv\left\{ \begin{array}{cccc}
0 & \text{mod }p &  & \text{if }e_{i}<e_{j}\\
0 & \text{mod }p^{e_{i}-e_{j}+2} &  & \text{if }e_{i}\ge e_{j}
\end{array}\right\} .
\]

Recall that $\mu(x,y)=z_{1}^{\alpha_{1}}z_{2}^{\alpha_{2}}\ldots z_{n}^{\alpha_{n}}$,
and the left-hand side of (\ref{eq:Star}) is $\mu(x,y)\theta(\mu(x,y)^{-1})$.

The left-hand side of (\ref{eq:Star}) is now
\[
(z_{1}^{-\alpha_{1}s_{1}p^{r_{1}}}\ldots z_{n}^{-\alpha_{n}s_{n}p^{r_{n}}})\times\left[z_{1}^{-(a_{12}\alpha_{2}+\ldots+a_{1n}\alpha_{n})}\right]\ldots\left[z_{n}^{-(a_{n1}\alpha_{1}+\ldots+a_{n,n-1}\alpha_{n-1})}\right].
\]

$\ $

Now we set up the preliminaries for constructing $\chi$. We note
that $t(x)^{p}\in Z$, as $G$ is $p$-central. So we may write 
\[
t(x)^{p}=z_{1}^{\beta_{1}}\ldots z_{n}^{\beta_{n}}
\]
 for some $\beta_{i}\in\nicefrac{\mathbb{Z}}{p^{e_{i}}\mathbb{Z}}$.
Similarly we have 
\[
t(y)^{p}=z_{1}^{\gamma_{1}}\ldots z_{n}^{\gamma_{n}}
\]
 for some $\gamma_{i}\in\nicefrac{\mathbb{Z}}{p^{e_{i}}\mathbb{Z}}$,
and 
\[
t(xy)^{-p}=z_{1}^{\delta_{1}}\ldots z_{n}^{\delta_{n}}
\]
 for some $\delta_{i}\in\nicefrac{\mathbb{Z}}{p^{e_{i}}\mathbb{Z}}$.

Again we consider equation (\ref{eq:dagger}). In terms of $z_{1},\ldots,z_{n},$
we have 
\[
z_{1}^{\alpha_{1}p^{2}}\ldots z_{n}^{\alpha_{n}p^{2}}=(z_{1}^{\beta_{1}p}\ldots z_{n}^{\beta_{n}p})(z_{1}^{\gamma_{1}p}\ldots z_{n}^{\gamma_{n}p})(z_{1}^{\delta_{1}p}\ldots z_{n}^{\delta_{n}p})
\]
\[
=z_{1}^{(\beta_{1}+\gamma_{1}+\delta_{1})p}\ldots z_{n}^{(\beta_{n}+\gamma_{n}+\delta_{n})p}.
\]
We note that for each $i=1,\ldots,n$, 
\begin{equation}
\alpha_{i}p^{2}=(\beta_{i}+\gamma_{i}+\delta_{i})p+k_{i}p^{e_{i}}\label{eq:(*&)}
\end{equation}
for some $k_{i}\in\mathbb{Z}$.

We construct $\chi$, which is dependent on $\beta_{i},\gamma_{i},\delta_{i}$,
as the composition of the two maps below:

\[
x\longmapsto t(x)^{p}=z_{1}^{\beta_{1}}\ldots z_{n}^{\beta_{n}},
\]
and

\[
z_{1}^{\beta_{1}}\ldots z_{n}^{\beta_{n}}\longmapsto(z_{1}^{\beta_{1}s_{1}p^{r_{1}-1}}\ldots z_{n}^{\beta_{n}s_{n}p^{r_{n}-1}})\times\left[z_{1}^{(\frac{a_{12}}{p}\beta_{2}+\ldots+\frac{a_{1n}}{p}\beta_{n})}\right]\ldots\left[z_{n}^{(\frac{a_{n1}}{p}\beta_{1}+\ldots+\frac{a_{n,n-1}}{p}\beta_{n-1})}\right].
\]

Note again that $\chi$ is defined on $Q$.

Using (\ref{eq:(*&)}), we check that the right-hand side of (\ref{eq:Star})
matches the previously computed left-hand side.

The right-hand side of (\ref{eq:Star}) is
\[
\chi(x)^{-1}\chi(y)^{-1}\chi(xy)
\]
$\ $
\[
=(z_{1}^{-\beta_{1}s_{1}p^{r_{1}-1}}\ldots z_{n}^{-\beta_{n}s_{n}p^{r_{n}-1}})\times[z_{1}^{-(\frac{a_{12}}{p}\beta_{2}+\ldots+\frac{a_{1n}}{p}\beta_{n})}]\ldots[z_{n}^{-(\frac{a_{n1}}{p}\beta_{1}+\ldots+\frac{a_{n,n-1}}{p}\beta_{n-1})}]\times
\]
\[
(z_{1}^{-\gamma_{1}s_{1}p^{r_{1}-1}}\ldots z_{n}^{-\gamma_{n}s_{n}p^{r_{n}-1}})\times[z_{1}^{-(\frac{a_{12}}{p}\gamma_{2}+\ldots+\frac{a_{1n}}{p}\gamma_{n})}]\ldots[z_{n}^{-(\frac{a_{n1}}{p}\gamma_{1}+\ldots+\frac{a_{n,n-1}}{p}\gamma_{n-1})}]\times
\]
\[
(z_{1}^{-\delta_{1}s_{1}p^{r_{1}-1}}\ldots z_{n}^{-\delta_{n}s_{n}p^{r_{n}-1}})\times[z_{1}^{-(\frac{a_{12}}{p}\delta_{2}+\ldots+\frac{a_{1n}}{p}\delta_{n})}]\ldots[z_{n}^{-(\frac{a_{n1}}{p}\delta_{1}+\ldots+\frac{a_{n,n-1}}{p}\delta_{n-1})}]
\]
$\ $
\[
=(z_{1}^{-(\beta_{1}+\gamma_{1}+\delta_{1})s_{1}p^{r_{1}-1}}\ldots z_{n}^{-(\beta_{n}+\gamma_{n}+\delta_{n})s_{n}p^{r_{n}-1}})\times
\]
\[
[z_{1}^{-(\frac{a_{12}}{p}(\beta_{2}+\gamma_{2}+\delta_{2})+\ldots+\frac{a_{1n}}{p}(\beta_{n}+\gamma_{n}+\delta_{n})}]\ldots[z_{n}^{-(\frac{a_{n1}}{p}(\beta_{1}+\gamma_{1}+\delta_{1})+\ldots+\frac{a_{n,n-1}}{p}(\beta_{n-1}+\gamma_{n-1}+\delta_{n-1})}].
\]
$\ $

Substituting (\ref{eq:(*&)}) $\beta_{i}+\gamma_{i}+\delta_{i}=\alpha_{i}p-k_{i}p^{e_{i}-1}$
into the above gives the following.
\[
(z_{1}^{-(\alpha_{1}p-k_{1}p^{e_{1}-1})s_{1}p^{r_{1}-1}}\ldots z_{n}^{-(\alpha_{n}p-k_{n}p^{e_{n}-1})s_{n}p^{r_{n}-1}})\times
\]
\[
[z_{1}^{-\frac{a_{12}}{p}(\alpha_{2}p-k_{2}p^{e_{2}-1})+\ldots+\frac{a_{1n}}{p}(\alpha_{n}p-k_{n}p^{e_{n}-1})}]\ldots[z_{n}^{-\frac{a_{n1}}{p}(\alpha_{1}p-k_{1}p^{e_{1}-1})+\ldots+\frac{a_{n,n-1}}{p}(\alpha_{n-1}p-k_{n-1}p^{e_{n-1}-1})}].
\]
$\ $

We simplify the above using the following facts:

- $o(z_{i})=p^{e_{i}}$ and $r_{i}\ge2$;

- $a_{ij}p^{e_{j}-2}\equiv0$ mod $p^{e_{i}}$ for $i\ne j$. 

So the right-hand side of (\ref{eq:Star}) is now
\[
(z_{1}^{-\alpha_{1}s_{1}p^{r_{1}}}\ldots z_{n}^{-\alpha_{n}s_{n}p^{r_{n}}})\times\left[z_{1}^{-(a_{12}\alpha_{2}+\ldots+a_{1n}\alpha_{n})}\right]\ldots\left[z_{n}^{-(a_{n1}\alpha_{1}+\ldots+a_{n,n-1}\alpha_{n-1})}\right]
\]
and this matches the left-hand side of (\ref{eq:Star}), as required.
\end{proof}
$\ $

Now, we calculate all such matrices in $R_{p}$ satisfying conditions
(a) to (c) in Proposition \ref{pro:(a),(b),(c)}, as these extend
to distinct elements in $\text{Aut}^{\nicefrac{G}{Z}}(G)$.

For the diagonal entries we have $|\nicefrac{p^{2}\mathbb{Z}}{p^{e_{1}}\mathbb{Z}}|\times\ldots\times|\nicefrac{p^{2}\mathbb{Z}}{p^{e_{n}}\mathbb{Z}}|=\frac{|Z|}{p^{2n}}$
choices. 

When $e_{i}<e_{j}$, we have $p^{e_{i}-1}$ choices as any element
of $\nicefrac{p\mathbb{Z}}{p^{e_{i}}\mathbb{Z}}$ works. When $e_{i}=e_{j}$
and $i\ne j$, we have $p^{e_{i}-2}$ choices as any element of $\nicefrac{p^{2}\mathbb{Z}}{p^{e_{i}}\mathbb{Z}}$
works. For each row $i$, we have $n-C_{i}$ off-diagonal entries
which correspond to $e_{i}\le e_{j}$. Of these $C_{i+1}-C_{i}-1$
correspond to $e_{i}=e_{j}$ and $n-C_{i+1}+1$ correspond to $e_{i}<e_{j}$.
So the number of choices for these entries is
\[
\prod_{i=1}^{l}(p^{e_{i}'-1})^{(n-C_{i+1}+1)(C_{i+1}-C_{i})}\prod_{i=1}^{l}(p^{e_{i}'-2})^{(C_{i+1}-C_{i}-1)(C_{i+1}-C_{i})}.
\]

When $e_{i}>e_{j}$, there are $p^{e_{j}-2}$ choices as any element
of $\nicefrac{p^{e_{i}-e_{j}+2}\mathbb{Z}}{p^{e_{i}}\mathbb{Z}}$
works. For each column $j$, there are $n-D_{j}$ entries corresponding
to $e_{i}>e_{j}$. So the number of choices for these entries is
\[
\prod_{j=1}^{l}(p^{e_{j}'-2})^{(n-D_{j})(C_{j+1}-C_{j})}=\prod_{j=1}^{l}(p^{e_{j}'-2})^{(n-C_{j+1}+1)(C_{j+1}-C_{j})}.
\]

This enables us to prove the following lemma.
\begin{lem}
\textup{\label{lem:OutG} Using the notation from before, for $e_{1}\ge2$,
\[
|\text{Aut}^{\nicefrac{G}{Z}}(G)|_{p}\ge\frac{|Z|}{p^{2n}}\prod_{i=1}^{l}(p^{e_{i}'-1})^{(n-C_{i+1}+1)(C_{i+1}-C_{i})}\prod_{i=1}^{l}(p^{e_{i}'-2})^{(n-C_{i})(C_{i+1}-C_{i})}.
\]
}

Furthermore, the non-trivial automorphisms calculated above are all
\emph{non-inner} automorphisms.\end{lem}
\begin{proof}
The number of extensions $\widetilde{\theta}\in\text{Aut}^{\nicefrac{G}{Z}}(G)$
from Proposition \ref{pro:(a),(b),(c)} is
\[
\frac{|Z|}{p^{2n}}\prod_{i=1}^{l}(p^{e_{i}'-1})^{(n-C_{i+1}+1)(C_{i+1}-C_{i})}\prod_{i=1}^{l}(p^{e_{i}'-2})^{(C_{i+1}-C_{i}-1)(C_{i+1}-C_{i})}\prod_{i=1}^{l}(p^{e_{i}'-2})^{(n-C_{i+1}+1)(C_{i+1}-C_{i})},
\]
which simplifies to
\[
\frac{|Z|}{p^{2n}}\prod_{i=1}^{l}(p^{e_{i}'-1})^{(n-C_{i+1}+1)(C_{i+1}-C_{i})}\prod_{i=1}^{l}(p^{e_{i}'-2})^{(n-C_{i})(C_{i+1}-C_{i})}.
\]
Therefore 
\[
|\text{Aut}^{\nicefrac{G}{Z}}(G)|\ge\frac{|Z|}{p^{2n}}\prod_{i=1}^{l}(p^{e_{i}'-1})^{(n-C_{i+1}+1)(C_{i+1}-C_{i})}\prod_{i=1}^{l}(p^{e_{i}'-2})^{(n-C_{i})(C_{i+1}-C_{i})}.
\]
It is clear that all the automorphisms $\widetilde{\theta}$ as in
Proposition \ref{pro:(a),(b),(c)} are non-inner, since $\widetilde{\theta}$
acts non-trivially on $Z$.

It remains to show that these non-inner automorphisms have order a
power of $p$. Denote by $P$ this finite set of non-inner automorphisms;
more precisely,
\[
P=\{\widetilde{\theta}\in\text{Aut}^{\nicefrac{G}{Z}}(G)\ |\ \theta:=\widetilde{\theta}|_{Z}\text{ satisfies (a) to (c) of Proposition \ref{pro:(a),(b),(c)}}\}.
\]
It is sufficient to show that $P$ is a subgroup, as then it follows
that every element of $P$ has $p^{\text{th}}$ power order since
$P$ is a $p$-group. 

To prove that we have a subgroup, we need to show that $P$ is multiplicatively
closed. That is, for $\widetilde{\theta}_{1},\widetilde{\theta}_{2}\in P$,
the composite $\widetilde{\theta}_{1}\cdot\widetilde{\theta}_{2}\in P$.

$\ $

It is enough to consider the restriction to $\text{Aut}(Z)$ since
$P$ is characterized by conditions (a) to (c) on $\text{Aut}(Z)$.
We have $\theta_{1},\theta_{2}\in\text{Aut}(Z)$ such that $\widetilde{\theta}_{1}|_{Z}=\theta_{1}$
and $\widetilde{\theta}_{2}|_{Z}=\theta_{2}$. Working with the matrix
representations, we note that $\theta_{1},\theta_{2}$ satisfy conditions
(a) to (c) of Proposition \ref{pro:(a),(b),(c)}. Let $\varphi(\theta_{1})=(a_{ij})$
and $\varphi(\theta_{2})=(b_{ij})$. Then $\varphi(\theta_{1}\cdot\theta_{2})=\varphi(\theta_{1})\varphi(\theta_{2})=(c_{ij})$
where $c_{ij}=\sum_{k}a_{ik}b_{kj}$. It is immediate that $\varphi(\theta_{1})\cdot\varphi(\theta_{2})\in R_{p}$
since $R_{p}$ is a ring. So (a) is satisfied for $\theta_{1}\cdot\theta_{2}$.
Using the expression for $c_{ij}$, it is clear that (b) is satisfied
for $\theta_{1}\cdot\theta_{2}$. 

For (c), we consider $c_{ij}$ for three cases: (1) $i<j$, (2) $i=j$
and (3) $i>j$.

\uline{Case (1).} We have $i<j$ and so $e_{i}\le e_{j}$. We write
\[
c_{ij}=\sum_{k\le i}a_{ik}b_{kj}+\sum_{i<k<j}a_{ik}b_{kj}+\sum_{k\ge j}a_{ik}b_{kj}.
\]

If $e_{i}<e_{j}$, we need to show that $c_{ij}\equiv0$ mod $p$.
As $p|a_{ij}$ and $p|b_{ij}$ for $i\ne j$, it is straightforward
that $c_{ij}\equiv0$ mod $p$.

If $e_{i}=e_{j}$, we need to show that $c_{ij}\equiv0$ mod $p^{2}$.
Again as $p|a_{ij}$ and $p|b_{ij}$ for $i\ne j$, we have that
\[
c_{ij}\equiv a_{ii}b_{ij}+a_{ij}b_{jj}\text{ mod }p^{2}.
\]
We further have that $p^{2}|a_{ij}$ and $p^{2}|b_{ij}$ since $e_{i}=e_{j}$.
So $c_{ij}\equiv0$ mod $p^{2}$, as required.

\uline{Case (2).} We have $i=j$, and so
\[
c_{ii}=\sum_{k<i}a_{ik}b_{ki}+a_{ii}b_{ii}+\sum_{k>i}a_{ik}b_{ki}.
\]
We need to show that $c_{ii}\equiv1$ mod $p^{2}$. As before, $c_{ii}\equiv a_{ii}b_{ii}$
mod $p^{2}$. Since $a_{ii}\equiv1$ mod $p^{2}$ and $b_{ii}\equiv1$
mod $p^{2}$, we have that $c_{ii}\equiv1$ mod $p^{2}$, as required. 

\uline{Case (3).} Here $i>j$ and hence $e_{i}\ge e_{j}$.We write
\[
c_{ij}=\sum_{k<j}a_{ik}b_{kj}+a_{ij}b_{jj}+\sum_{j<k<i}a_{ik}b_{kj}+a_{ii}b_{ij}+\sum_{k>i}a_{ik}b_{kj}.
\]
For $k<j$, we have $p^{e_{i}-e_{j}+2}$ divides $p^{e_{i}-e_{k}+2}$,
which in turn divides $a_{ik}$. Similarly for $k>i$, we have $p^{e_{i}-e_{j}+2}$
divides $p^{e_{k}-e_{j}+2}$, which divides $b_{kj}$. For $k=j$,
we have $p^{e_{i}-e_{j}+2}$ divides $a_{ij}$. Similarly for $k=i$,
we have $p^{e_{i}-e_{j}+2}$ divides $b_{ij}$. For $j<k<i$, we have
$p^{e_{i}-e_{k}+2}$ divides $a_{ik}$ and $p^{e_{k}-e_{j}+2}$ divides
$b_{kj}$. So $p^{e_{i}-e_{j}+4}$ divides $a_{ik}b_{kj}$. Therefore
$c_{ij}\equiv0$ mod $p^{e_{i}-e_{j}+2}$ as required. 

So (c) is satisfied for $\theta_{1}\cdot\theta_{2}$. Thus $\widetilde{\theta}_{1}\cdot\widetilde{\theta}_{2}\in P$.

Therefore, $P$ is a subgroup as required.
\end{proof}
$\ $

\textbf{PROOF OF THEOREM \ref{thm:Conj}.}

We recall that we need $|\text{Out}(G)|_{p}\ge|Z|$ to prove the theorem,
we now analyse our lower bound for $|\text{Out}(G)|_{p}$, as given
in Lemma \ref{lem:OutG}. As $e_{1}>2$, we have
\[
|\text{Out}(G)|_{p}\ge\frac{|Z|}{p^{2n}}\prod_{i=1}^{l}(p^{2})^{(n-C_{i+1}+1)(C_{i+1}-C_{i})}\prod_{i=1}^{l}p{}^{(n-C_{i})(C_{i+1}-C_{i})}
\]
\[
\qquad\qquad\qquad\qquad\qquad\qquad\quad=\frac{|Z|}{p^{2n}}\prod_{i=1}^{l}p^{(n-C_{i+1}+1)(C_{i+1}-C_{i})}\prod_{i=1}^{l}p^{(n-C_{i+1}+1)(C_{i+1}-C_{i})}\prod_{i=1}^{l}p^{(n-C_{i})(C_{i+1}-C_{i})}
\]
\[
\qquad\qquad\qquad\quad=\frac{|Z|}{p^{2n}}\prod_{i=1}^{l}p^{(n-C_{i+1}+1)(C_{i+1}-C_{i})}\prod_{i=1}^{l}p^{(C_{i+1}-C_{i})[2n+1-(C_{i+1}+C_{i})]}
\]
\[
\ge\frac{|Z|}{p^{2n}}\prod_{i=1}^{l}p^{(C_{i+1}-C_{i})(2n+1)-(C_{i+1}^{2}-C_{i}^{2})}\qquad
\]
\[
=\frac{|Z|}{p^{2n}}p^{(2n+1)(C_{l+1}-C_{1})-(C_{l+1}^{2}-C_{1}^{2})}\qquad\quad
\]
\[
=|Z|p^{n^{2}-3n}.\qquad\qquad\qquad\qquad\qquad\ \ 
\]
As $n\ge3$, we have that $|\text{Out}(G)|_{p}\ge|Z|$ as required.$\qquad\qquad\qquad\qquad\qquad\qquad\qquad\square$

$\ $

\end{document}